\documentclass[12pt]{article}
\usepackage{layout,verbatim,amssymb,amsmath,setspace,cite}
\usepackage{amsthm,authblk}
\newcommand{\erre}{\mbox{$\mathbb{R}$}}
\newcommand{\enne}{\mbox{$\mathbb{N}$}}
\newcommand{\iE}{\mbox{$\displaystyle{\int_{E}}$}}

\newcommand{\pgn}{{P}_{\gamma_n}}
\newcommand{\g}{\mbox{$i(F)$}}
\providecommand{\keywords}[1]{\textbf{{Key Words:}} #1}
\providecommand{\msc}[1]{\textbf{{M.S.C.:}} #1}
\newtheorem{theorem}{Theorem}[section]
\newtheorem{lemma}[theorem]{Lemma}
\newtheorem{proposition}[theorem]{Proposition}
\newtheorem{corollary}[theorem]{Corollary}
\newtheorem{definition}{Definition}[section]
\newtheorem{remark}[theorem]{Remark}
\newtheorem{example}[theorem]{Example}
\numberwithin{equation}{theorem}
\title{\bf Henstock  multivalued integrability  in Banach lattices with respect to pointwise non atomic measures}
\author{Antonio Boccuto,
       Domenico Candeloro,
        Anna Rita Sambucini}
\affil{Department of Mathematics and Computer Sciences,\\ University of Perugia,\\
 1, Via Vanvitelli -- 06123, Perugia (Italy)\\ 
\mbox{~} \\
\small emails: antonio.boccuto@unipg.it,domenico.candeloro@unipg.it, anna.sambucini@unipg.it}
\begin{document}
\date{}
\maketitle

\abstract{Henstock-type integrals are considered, for multifunctions taking values 
in the family of weakly compact and convex subsets of  a Banach lattice $X$.
The main tool to handle the multivalued case is a R{\aa}dstr\"{o}m-type embedding theorem established 
by C. C. A. La\-buscha\-gne,  A. L. Pinchuck, C. J. van Alten in 2007.
In this way the norm and order integrals reduce
 to that of a single-valued function taking values in an $M$-space, and new 
proofs are  deduced for some decomposition results recently stated in 
two recent papers by Di Piazza and Musia\l \, based on the existence of integrable selections.}\\

\noindent
\\

\noindent
\msc{28B05, 28B20, 46B42, 46G10, 18B15}
%====================================================
\section{Introduction}

After the pioneering works of Aumann  and Debreu, the multivalued case has been intensively studied, several notions of integral for
multivalued functions  in Banach and other vector spaces have been developed.
These notions have shown to be useful when modelling some theories in different fields as  optimal
control and mathematical economics, see for example \cite{dps1,cc}. 
The choice to deal with these types of integration is motivated by the fact that  Bochner integrability of  selections 
is a strong condition; moreover selection theorems for the Auman--Bochner integral are stated
in the separable context. 
In order to overcome this problem  contributions  have been given also in \cite{cascales2007a, cascales2010, DPVM,cg2014}.

The starting point of this research are  the papers \cite{BS2004,BS2011,bms, cs2014} in which the comparison between norm and order-(multi) valued integration is examined in Banach lattices and vector lattices. 
 Another fundamental tool of this paper is the embedding theorem given  in \cite{L1}, in which a vector lattice version of the R{\aa}dstr\"{o}m  embedding theorem is stated.
As far as we know a first result which uses an embedding theorem was given in \cite{ab} and then by Debreu, Castaing and his school (see for example \cite{CV}) and later on by many other authors: see for example, 
\cite{bas2004,ms2001,ms2002,BS2004,CASCALES2,cascales2007a,Lm,dpm3,dpm4,dpmb1,DPM,DPVM}.
This paper aims 
to provide a comparison between different types of multivalued integration, in  Banach lattices.\\
Throughout this paper $(T,d)$ is a compact metric Hausdorff topological space, $\Sigma$ its Borel
$\sigma$-algebra and $\mu : \Sigma \to \mathbb{R}_0^+$ a regular, pointwise non atomic measure,
 $X$ a  Banach (lattice) space and $cwk(X)$ the family of all convex weakly compact non-empty subsets of $X$. \\
 In Section \ref{multi}  the Henstock multivalued integral is introduced (H-integral for multifunctions).
A very important tool for the study of set-valued integration is the Kuratowski
and Ryll-Nardzewski's theorem which guarantees the existence of measurable selectors, though it has the
handicap of the requirement of separability for the range space. Recently this result was extended to the non separable case in \cite{cascales2007a,cascales2010,dpmb1} for other kinds of multivalued integrals or for the $ck(X)$-valued case,
while here  a selection theorem is given for $cwk(X)$-valued, H-integrable multifunctions.

Moreover, thanks to the structure of near vector space of $cwk(X)$, it is proved that the  H-integrability of $F$ is  equivalent to the H-integrability (and also order-type integrability) of its {\em embedded} function $i(F)$,  and that the multivalued integral here obtained coincides with the one given in \cite{BS2004}.
This fact also implies that the H-integral of $F$ and the norm-integral  $\Phi(F, \cdot)$ given in \cite{BS2004} are countably additive multimeasures.
Of course, if the target space is a general Banach lattice (not necessarily an $M$-space), considering order convergence gives rise to a new concept of integral. 
In Subsection \ref{o-multi}, indeed,  the order structure is considered,  since vector lattices play an important role  (see for example \cite{bms,bvl,bvg,bbs,dallas,csmed,mn,L3,cv2014,av2010}) also in applications, and an H-(multivalued) integral with respect to the order is studied.\\
%=====================================================================================
%=====================================================================================
\section{Preliminaries on Henstock and McShane  integrals}\label{HMS-univoco}
Given a compact metric Hausdorff topological space $(T,d)$ and its Borel
$\sigma$-algebra $\Sigma$, let $\mu: \Sigma \rightarrow \mathbb{R}_0^+$ be a 
 pointwise non atomic (namely $\mu (\{t\}) = 0$ for every $t \in T$)
 $\sigma$-additive (bounded) regular measure, 
so that $(T,d,\Sigma,\mu)$ is a Radon measure space. Let $X$ be a Banach space.
The following concept of Henstock integrability was presented in \cite{f1994a} for bounded measures in the Banach space context. See also  \cite{bms} for the following definitions and investigations.

A \textit{gauge} is any map $\gamma: T \rightarrow \mathbb{R}^+$.
A \textit{decomposition}  $\Pi$  of $T$ is a finite family $\Pi = \{ (E_i,t_i): i=1, \ldots,k \} $ of pairs such that
$t_i \in T, \, E_i \in \Sigma$ and $\mu(E_i\cap E_j) = 0$ for $i \neq j$.
 The points $t_i$, $i=1, \ldots, k$, are called \textit{tags}.
If moreover $\bigcup_{i=1}^k \, E_i = T$, $\Pi$ is called a \textit{partition}.\\
A \textit{Perron partition} is a partition such that $t_i \in E_i$ for every $i$.\\
Given a gauge $\gamma$, any partition $\Pi$ is said to be 
 \textit{$\gamma$-fine} $(\Pi \prec \gamma)$ if
$d(w,t_i) < \gamma (t_i)$ for every $ w \in E_i$ and $i = 1, \ldots, k$.
Equivalently a gauge $\gamma$ can be also defined as a mapping associating with each point $t\in T$ an open set which contains the point $t$.
\begin{definition}\rm  \label{fnorm}
\rm A function $f:T\rightarrow X$ is \textit{{\rm Ms}-integrable} (resp. \textit{{\rm H}-integrable})
if there exists $I \in X$ such that, for every $\varepsilon > 0$
 there is a gauge $\gamma : T \rightarrow \mathbb{R}^+$ such that for every $\gamma$-fine partition (resp. Perron partition)
of $T$, 
$\Pi=\{(E_i, t_i), i=1, \ldots, q \}$, one has:
\begin{eqnarray*}
\left\|\sigma(f, \Pi)- I \right\| \leq \varepsilon,
\end{eqnarray*}
where the symbol $\sigma( f, \Pi)$ means
 $\sum_{i=1}^q  f(t_i) \mu(E_i)$.
In this case, the following notation will be used: $I = {\rm (Ms)}\int_T f d\mu$ 
(resp. $I = {\rm (H)}\int_T f d\mu$).
\end{definition}

\begin{remark}\label{ereditarieta} \rm
It is not difficult to deduce, in case $f$ is H-integrable in the set $T$, that also the restrictions $f\ 1_E$ are, for every measurable set $E$, thanks to the Cousin Lemma (see \cite[ Proposition 1.7]{riecan}).
This is not true in the classical theory, where  $T=[0,1]$ and the partitions allowed are only those consisting of sub-intervals.
 As it is well-known, in such classical case H-integrability is different from McShane integrability and also from Pettis integrability.  Nevertheless, McShane integrability of a mapping $f:[0,1]\to X$ (here $X$ is any Banach space) is equivalent to Henstock and Pettis simultaneous integrabilities to hold (see \cite[Theorem 8]{f1994a}).
\end{remark} 

Since the measure $\mu$ is assumed to be pointwise  non atomic,  all 
the concepts in the Henstock sense turn out to be equivalent to the same concepts in the McShane sense (i.e. without requiring that the {\em tags} are contained in the corresponding sets of the involved partitions)
as shown in \cite[Proposition 2.3]{cs2015}, in fact  it will be sufficient to observe that, for every  gauge $\gamma$ and any $\gamma$-fine partition  $$\Pi^0:=\{(B_i,t_i): i=1,...,k\},$$
 there exists a Henstock-type $\gamma$-fine partition $\Pi^{\prime}$ satisfying $\sigma(f,\Pi^0) = \sigma(f,\Pi^{\prime})$.
 Without loss of generality,  all the tags $t_i$ can be supposed to be distinct.
Now, set $A:=\{t_i,i=1,...,k\}$ and define, for each $j$:
$$B^{\prime}_j:=(B_j\setminus A)\cup\{t_j\}.$$
 Then the pairs $(B_j,t_j)$ form a $\gamma$-fine Henstock-type partition $\Pi^{\prime}$ and $\mu(B^{\prime}_j)=\mu(B_j)$ for all $j$, so
$\sigma(f,\Pi^0) = \sigma(f,\Pi^{\prime})$.
So the Perron condition will no longer be necessary.\\

For the sake of clearness, it is useful to remark that actually 
McShane integrability in the classical case for a mapping $f:[0,1]\to X$ is equivalent to the present
 definition of Ms-integrability (see \cite{f1995}), while Henstock integrability (in the classic case) is strictly weaker than McShane. 
However, since in  this context (for pointwise non atomic measures) the two notions coincide, 
they will be no more distinguished, and will be simply referred as the {\em Henstock} integral.\\

From now on  suppose that   $X$ is a Banach lattice,
$X^+$ is its positive cone and $X^{++}$ is the subset of strictly positive elements of $X$.
A sequence $(p_n)_n$ in $X$ is called \textit{$(o)$-sequence} iff it is decreasing and $\wedge_n p_n=0$.\\
The symbols $|~ |,\, \| ~\|$ refer to modulus and norm of $X$, respectively; for the relation between them see for example
\cite{bvg,bvl,fvol3,bms}.\\
An element $e$ of $X^+$ is an order unit in $X$ if $X$ is the solid linear subspace of itself generated by $e$;
that is, if for every $u \in X$ there is an $n \in \mathbb{N}$ such that $|u| \leq  ne$.\\ 
An $M$-space is a Banach lattice $M$ in which the norm is an order-unit norm, namely 
there is an order unit $e$ in $M$ and an equivalent Riesz norm $\| \cdot \|_e $ defined on $M$ by the formula
$\|u\|_e := \min\{ \alpha: |u| < \alpha e \}$
for every $u \in M$. In this case one has also
$\|u+v\|=\|u\|\vee \|v\|$ for positive $u,v\in M$.\\
An $L$-space is a Banach lattice $L$ such that $\|u + v\| = \|u\| + \|v\|$ whenever $ u, v\in  L^+$.
 For further details about $L$- and $M$-spaces see also \cite{fvol3}.  \\

In the setting of Banach lattices the notion of {\em order-type} integral can be given as follows.
\begin{definition}\rm  \label{forder}
\rm A function $f:T\rightarrow X$ is \textit{order-integrable}
in the Henstock sense (\textit{{\rm (oH)}-integrability}  for short) if there exists $J \in X$ together with an $(o)$-sequence $(b_n)_n$ in $X$ 
and a corresponding sequence $(\gamma_n)_n$ of gauges, such that for every $n$ and every $\gamma_n$-fine
 (Perron) partition of $T$,
$\Pi=\{(E_i, t_i), i=1, \ldots, q \}$, 
it is
$\left|\sigma(f, \Pi)- J \right| \leq b_n$;
then the integral $J$ will be denoted with (oH)$\int f$.
\end{definition}
Also in this case there is no difference in taking $\gamma_n$-fine arbitrary
 partitions rather than Perron-type partitions.
So from now on, unless otherwise specified,   only arbitrary
 $\gamma_n$-fine partitions will be considered.
%===================================================
\section{Gauge multivalued integrals}\label{multi}
The previous notions of integral can be extended to the case of multivalued mappings. 
For the norm integral, the concept given in \cite{BS2011} will be adopted here, comparing it with other definitions: thanks to a well-known R{\aa}dstr\"{o}m-type embedding theorem  given in \cite{L1}, the multivalued case can be reduced to the single-valued one, which can help to study the integral.\\
When the space $X$ is a Banach lattice, also the notion of {\em order}-type integral will be introduced, for multivalued mappings, essentially following the lines of the paper \cite{bms}. Also for this notion, some properties and comparisons will be stated.\\

Let $P_0(X)$  be  the set of all nonempty subsets of the Banach
 space $X$, and  
 $bf(X)$  be  the set of all nonempty, bounded and closed subsets of $X$; moreover, let 
$cwk(X)$, $ck(X)$, denote respectively the subfamilies of 
$bf(X)$ of all   
weakly compact, convex; 
 compact, convex subsets of $X$.\\
 For all $A,B \in P_0(X)$ and $\lambda \in \mathbb{R}$  the Minkowski addition and scalar multiplication are defined as
\begin{eqnarray*}\label{operations}
A + B = \{ a+b: a \in A, b \in B \}, \,\,\, \mbox{and} \,\,\,\, 
\lambda A =  \{ \lambda a : a \in A. \}
\end{eqnarray*}

As in \cite{CV,L1} the operation $\oplus$  on $bf(X)$ is defined by setting  $A \oplus B := \mbox{cl} (A +B)$. 
If $A_i \in bf(X), \,\, i = 1, \ldots, n$, then
\begin{eqnarray*}\label{sommapunto}
\sum_{i=1}^n A_i:= \mbox{cl}(A_1  +  \cdots  + A_n).
\end{eqnarray*}
Observe that if $A,B \in bf(X)$ then $A \oplus B \in bf(X)$. 
In case $A$ and $B$ are in $cwk(X)$ (or in $ck(X)$), then the Minkowski addition 
is already closed, so in these cases the closure in  the previous operation
is not needed. For all unexplained terminology on multifunctions we refer to \cite{CV}.
Since the set inclusion is a natural ordering on $P_0 (X)$, which is compatible with 
the Minkowski addition and scalar multiplication,
in \cite{L1} a R{\aa}dstr\"{o}m embedding theorem
 is extended to the space $cwk(X)$  as follows.

\begin{theorem}{\rm  (\cite[Theorem 5.6]{L1})} \label{5.6}
There exists a compact Hausdorff space $\Omega$ and a positively linear map 
$i:cwk(X) \to C(\Omega)$ such that
\begin{description}
\item[\ref{5.6}.1)] $d_H(A,C)=\|i(A)-i(C)\|_{\infty},\ \ A,C\in cwk(X)$ 
{\rm (of course, here $d_H$ denotes the Hausdorff distance (see \cite{L1,CV}))};
\item[\ref{5.6}.2)] $i(cwk(X))=cl(i(cwk(X)))\ \ \ \ {\rm (norm\ closure). \ }$
\item[\ref{5.6}.3)] $i(\overline{co}(A\cup C)) =\max\{i(A),i(C)\}$ for all $A,C$ in $cwk(X)$.
\end{description}
\end{theorem}
%==============================
\begin{comment}
\begin{theorem}{\rm  (\cite[Theorem 5.6]{L1})} \label{5.6}
Let $X$ be any Banach space, and 
denote by $S$ any of the spaces $cwk(X)$ or $ck(X)$. Then there exists a compact Hausdorff space $\Omega$ and a positively linear map $i:S\to C(\Omega)$ such that
\begin{description}
\item[\ref{5.6}.1)] $d_H(A,C)=\|i(A)-i(C)\|_{\infty},\ \ A,C\in S$ 
{\rm (of course, here $d_H$ denotes the Hausdorff distance (see \cite{L1,CV}))};
\item[\ref{5.6}.2)] $i(S)=cl(i(S))\ \ \ \ {\rm (norm\ closure). \ }$
\item[\ref{5.6}.3)] $i(\overline{co}(A\cup C)) =\max\{i(A),i(C)\}$ for all $A,C$ in $S$.
\end{description}
\end{theorem}
\end{comment}
%===================================

\subsection{Norm multivalued integrals}\label{n-multi}
In this section a short survey is presented of the different notions of the multivalued integral in the norm sense, in order to compare them and possibly 
improve some results.
\begin{definition}\rm
For a  multifunction \mbox{$F: T \rightarrow P_0 (X)$\,} let $S^1_{F,H}$ 
be the set of all {\rm H}-integrable selections of $F$ in the sense of Definition \ref{fnorm}, namely:
\( S^1_{F, H} =\{  f: f(t) \in F(t)
\hskip.1cm \mu-\mbox{a.e. and \,}  f \,\,  \mbox{is H-integrable.} \}\)
\end{definition}

\begin{definition} \label{Hnorma}\rm
If $S^1_{F, H}$ is non-empty, then for every $E \in \Sigma$   the {\em Aumann-Henstock integral~} 
((AH)-{\em integral}) of $F$  is defined as the set
\[ ({\rm AH})\iE F d \mu = \left\{  \iE f d\mu, f \in S^1_{F, H} \right\}.\]
\end{definition}

In order to prove existence of the previous selections, the following concept of integrability is also introduced, for a multifunction \mbox{$F:T\to cwk(X)$.}

\begin{definition}\label{normintegral}\rm
Given a multifunction $F:T\to cwk(X)$, $F$ is {\em {\rm H}-integrable} if there exists an element $J\in cwk(X)$, such that for every $\varepsilon>0$ there exists a gauge $\gamma$ such that, for every $\gamma$-fine partition $\Pi$, the following holds:
$d_H(\sum_{\Pi} F,J)\leq \varepsilon.$
In this case one writes
$J:={\rm (H)}\int_T F d\mu.$
\end{definition}

Also in this case, existence of the integral in $T$ implies existence in all measurable subsets $E$ of $T$ (which will be denoted by $J_E(F)$)
analogously to Remark \ref{ereditarieta}.
Indeed, as later shown, the last definition of integrability reduces to H-integrability for a corresponding single-valued function $F$ taking values in the space $C(\Omega)$, where $\Omega$ is a suitable compact space: see Theorems \ref{5.6} and \ref{nuovoconfronto} below.
\\
For the definitions of Mc Shane and Pettis integrals for multifunctions see for example \cite[Definition 2]{BS2004},\cite[Definition 3.1]{BS2011}, omitting the "limsup" in the first definition since here $T$ is compact and $\mu$ is bounded.
One can observe also that, in case a multifunction $F:T\to cwk(X)$ is H-integrable, then there exist H-integrable selections.
 This result has been proved in \cite[Theorem 2.5]{cascales2007a} for Pettis integrable selections and  in 
 \cite[Theorem 3.1]{dpm4}, but for Henstock integrable multivalued mappings defined in the real interval $[0,1]$.

\begin{theorem}\label{esistenza}
If $F:T\to cwk(X)$ is {\rm H}-integrable, then 
$S^1_{F, H} \neq \emptyset$. 
\end{theorem}
\begin{proof}
 We first observe that, thanks to  H\"{o}rmander equality (see e.g. \cite[formula (3)]{dpm4}), the function $F$ is scalarly H-integrable, i.e. the {\em support} mappings $t\mapsto s(x^*,F(t))$ are $H$-integrable, for all $x^*\in X^*$.
Let us set $K:={\rm (H)}\int_T F d\mu$, and choose any strongly exposed point $x_0\in K$: such a point exists, since $K$  is in $cwk(X)$. Then there exists a functional $x_0^*\in X^*$ such  that $x^*_0(x)<x^*_0(x_0)$ for every $x\in K$, $x\neq x_0$.
Let 
$$G(t):=\{x\in F(t):x_0^*(x)=s(x^*_0,F(t))\},$$ for every $t\in T$.
Proceeding like in the proof of  \cite[Theorem 2.5]{cascales2007a},  
one has that $G$ is Pettis integrable  in the sense of \cite{cascales2007a}
(i.e. the {\em support mappings} $s(x^*,G)$ are integrable for every $x^*\in X^*$ and for every measurable subset $E\subset T$ there exists an element $P_E(G)\in cwk(X)$ (denoted also by $(P)\int_E G d\mu$) such that 
$$s(x^*,P_E(F))=\int_E s(x^*,G)d\mu),$$
 and has  a Pettis integrable selector $g$ (which is also therefore a selector of $F$),
for which $x^*_0(x_0)=\int_T x^*_0 g(t) d\mu$.  
Moreover
 $(P)\int_T G d\mu=\{x_0\}$, and
$x^*(g)=s(x^*,G)$ $\mu$-a.e. for all $x^*$, and every selector $g$ of $G$. 
So, for each $x^*$, the mapping $x^*(g)$ is Lebesgue-integrable, and therefore McShane-integrable (see \cite[Theorem 1O]{fvol4}), which in our setting means Henstock integrability. Then, proceeding as in the proof of  \cite[Theorem 3.1]{dpm4}, 
the conclusion follows, i.e. $g$ is H-integrable.
\end{proof}
The multivalued {\em norm} integral, introduced in \cite[Definition 3.13]{bms} will be now recalled. This definition is inspired at a similar notion introduced in \cite[Definition 3]{JK} in order to study differential relations.
In the subsection \ref{o-multi} a corresponding notion for the {\em order-type} multivalued integral will be introduced in the Banach lattice context and compared with this.
\begin{definition}\rm 
\cite[Definition 3.13]{bms}
\label{normultintegral}
Let $F:T \to P_0 (X)$ be a multifunction, and $E \in \Sigma$. 
We call \it $(\| \cdot \|)$-integral \rm of $F$ on $E$ the set
\begin{eqnarray*}
\Phi (F,E) &=& \left\{ \right. z \in X: \text{for every }  \,\, \varepsilon \in \mathbb{R}^+ 
 \text { there is a gauge } \,  \gamma : T \rightarrow \mathbb{R}^+ : 
\\ &&
\displaystyle{\inf_{c \in \sum_{\Pi_{\gamma}} \, F} \|z-c\|} \leq\varepsilon \,\,  \text{for every}\,\, \gamma\text{-fine partition }  \\ &&
 \, \left. 
\Pi_{\gamma}:= \{ (E_i,t_i): i=1, \ldots,k \}  
\ \text{of } E. \right\} 
\end{eqnarray*}

Alternatively, one can write (\cite[Proposition 3.7]{bms}) 
\begin{eqnarray}\label{capcup}
\Phi(F,E)=\bigcap_n\bigcup_{\gamma}\bigcap_{P_{\gamma,E}}
\left[\Sigma_{\Pi}\,F\oplus\frac{B_X}{n}\right],
\end{eqnarray}
where $P_{\gamma,E}$ is the family of all (Perron)
 $\gamma$-fine partitions of E. 
\end{definition}

\begin{remark}\label{unainclusione} \rm 
We collect here some relations between the previous definitions of integral, and also with the Aumann integral.
In case $F$ is $cwk(X)$-valued and H-integrable on $E$  there exists an element $J_E(F)\in cwk(X)$ such that, for every $n\in \enne$ it is possible to find a gauge $\gamma_n$ such that
\begin{eqnarray}\label{prima}
\sum_{\Pi} \, F \subset J_E(F)+n^{-1} B_X  \mbox{\,\, and\,\, }
J_E(F)\subset \sum_{\Pi} \, F + n^{-1}B_X
\end{eqnarray}
hold true, for every $\gamma_n$-fine partition $\Pi$, where $B_X$ is the unit ball in $X$.
So, in this case we have clearly
$J_E(F)\subset \Phi(F,E).$
Later also the converse inclusion will be proven (see Proposition \ref{immersa con Phi}).\\
Observe that, for every $E \in \Sigma$, the following inclusion holds:
$$({\rm AH})\iE F d \mu\subset \Phi(F,E)$$
since, if $f \in S^1_{F,H}$, then its H-integral belongs by definition to  the right member of (\ref{capcup}). However, in case $F$ is H-integrable, then 
$$({\rm AH})\iE F d \mu\subset J_E(F):$$
indeed, if $f\in S^1_{F,H}$, for each $n$ there exists a gauge $\gamma'_n$ such that 
$${\rm (H)}-\int_E f d\mu\in \Sigma_{\Pi}\, F\oplus n^{-1}B_X$$
for all $\gamma'_n$-fine partitions $\Pi$; then, if $\Pi$ is $(\gamma_n\wedge \gamma'_n)$-fine, 
$${\rm (H)}-\int_E f d\mu\in J_E(F)\oplus 2n^{-1}B_X.$$
(here $\gamma_n$ is the gauge corresponding to $n$ in the definition of $J_E(F)$). By arbitrariness of $n$ and closedness of $J_E(F)$, we get that 
$${\rm (H)}-\int_E f d\mu\in J_E(F);$$
 and, by arbitrariness of $f$, the announced inclusion follows.\\
Moreover since $f$ is H-integrable for every $E \in \Sigma$  
then Definition \ref{normintegral}
is equivalent to the ($\star$)-integral given in \cite[Definition 2]{BS2004}.\\
Assuming that $X$ is a separable Banach space and that there
exists a countable family $(x^{\prime}_n)_n$ in $X^{\prime}$
which separates points of $X$ then, thanks to  \cite[Theorem 1]{BS2004}
the following equalities follow,  for any measurable and integrably bounded multifunction $F$ with values in $cwk(X)$:
$$J_E(F)=({\rm AH})\iE F d \mu=\Phi(F,E).$$
\end{remark}
 
Using Theorem \ref{5.6},  according to \cite[Definition 2.1]{CASCALES2} and \cite[Definition 3]{BS2004}, the following main 
result can be stated, which is analogous to \cite[Proposition 4.4]{cascales2007a} (given for the Pettis integrability),
together with a  decomposition theorem that can be compared with previous ones obtained in \cite[Theorem 1]{dpmb1} and in  \cite[Corollary 3.2]{dpm4}. A definition is needed, however.

\begin{definition}\label{weak bochner}\rm
If $F:T\to cwk(X)$ is a multivalued mapping and $X$ any Banach space, we say that $F$ is {\em non-negative} if $i(F):T\to C(\Omega)$ is, where $i$ is the embedding found in Theorem \ref{5.6}.
\end{definition}
\begin{theorem}\label{nuovoconfronto}
Let $F:T\rightarrow cwk(X)$ be  a multifunction. The following are equivalent:
\begin{description}
\item[\ref{nuovoconfronto}.1)]
  $F$ is {\rm H}-integrable  {\rm(}in the sense of  Definition \ref{normintegral}{\rm)};
\item[\ref{nuovoconfronto}.2)]
 the {\em embedded} function $i(F):T\rightarrow C(\Omega)$ is {\rm H}-integrable, and in that case {\rm (H)}-$\int i(F)d\mu=i(J_T(F))$; 
\item[\ref{nuovoconfronto}.3)] for every $E \in \Sigma$ and  $\varepsilon > 0$ there exists a gauge $\gamma$ such that
$ \| \sigma( i(F),\Pi_1) - \sigma( i(F),\Pi_2) \|_{\infty} \leq \varepsilon $
for every   $\gamma$-fine partitions $\Pi_1, \Pi_2$ of $E$;
\item[\bf \ref{nuovoconfronto}.4)]\  $F$ is the sum of a non-negative {\rm H}-integrable multifunction $G$ with values in $cwk(X)$ and an {\rm H}-integrable single-valued function $f:T\to X$.
\end{description}
Moreover, the  previous statements  imply
\begin{description}
\item[\ref{nuovoconfronto}.5)] the family $W_F = \{s(x^*,F) : x^* \in B_{X^*}\}$ is uniformly integrable.
\end{description}
\end{theorem}
\begin{proof} 
Thanks to Theorem \ref{5.6}, it is easy to see that
{\bf\ref{nuovoconfronto}.1}) is equivalent to {\bf\ref{nuovoconfronto}.2}).
Obviously {\bf\ref{nuovoconfronto}.1}) implies {\bf \ref{nuovoconfronto}.3}).
We now prove that {\bf\ref{nuovoconfronto}.3}) implies {\bf \ref{nuovoconfronto}.2}), i.e.,
 if the Cauchy-Bolzano condition is true, then \g\ is H-integrable.
Taking $\varepsilon = n^{-1}$, then there exists a gauge $\gamma_n$  such that, 
for every pair $(\Pi_1^n, \Pi_2^n)$ of $\gamma_n$-fine partitions of $E$, 
 \[ \|\sigma(\g,\Pi^n_1) - \sigma(\g,\Pi^n_2) \|_{\infty} \leq n^{-1}.\]
Without loss of generality assume that $(\gamma_n)_n$ is decreasing in $n$. \\
Since  $C(\Omega)$ is  Dedekind complete (\cite[Proposition 1.2.4]{mn}), then
$\bigwedge_{\Pi \in \pgn} \sigma(\g,\Pi)$, 
$\bigvee_{\Pi \in \pgn} \sigma(\g,\Pi)$
 are in $C(\Omega)$ and for every $n \in \enne$ 
it is obvious that
\begin{eqnarray*}
\bigwedge_{\Pi \in \pgn} \sigma(\g, \Pi) \leq  \bigvee_{\Pi \in \pgn} \sigma(\g,\Pi).
\end{eqnarray*}
Let $z =  \bigvee_n  \bigwedge_{\Pi \in \pgn} \sigma( \g, \Pi).$ Then  it is easy to see that $z$ verifies the definition of integrability.
In order to prove that {\bf \ref{nuovoconfronto}.2}) implies {\bf\ref{nuovoconfronto}.5}), observe that the single valued function $i(F)$ is H-integrable 
%and then Mc Shane integrable since in this  setting ($\mu$ pointwise non atomic) the two definitions coincide. 
so $i(F)$ is Pettis integrable and then, by 
\cite[Proposition 4.4]{cascales2007a} $F$ is Pettis integrable and so the assertion follows from \cite[Theorem 4.1]{cascales2007a}. 
 \\
Also, it is obvious that {\bf \ref{nuovoconfronto}.4)} implies {\bf \ref{nuovoconfronto}.1)}. Let us prove now that
{\bf \ref{nuovoconfronto}.1)} implies {\bf \ref{nuovoconfronto}.4)}.\\ 
To this aim, let $f$ be any H-integrable selection of $F$ (see Theorem \ref{esistenza}), and 
define $G$ by translation: $F(t)=f(t)+G(t)$. This implies that $0\in G(t)$ for all $t$, and therefore $s(x^*,G(t))\geq 0$ for all $t$ and $x^*$. So the R{\aa}dstr\"{o}m
 embedding of $G(t)$ is a non-negative element of $l^{\infty}(B_{X^*})$; and, since the Kakutani isomorphism preserves order, also $i(G(t))$ is non-negative, and therefore $G$ is a non-negative $cwk(X)$-valued mapping. Integrability of $G$ follows immediately from the fact that $i(F)=i(f)+i(G)$, and linearity of the H-integral.
The last implication follows easily by  {\bf \ref{nuovoconfronto}.2}) and \cite[Theorem 4.1]{cascales2007a}.
\end{proof}

From now on in this subsection, all multivalued mappings will be assumed to be 
$cwk(X)$-valued.
Thanks to well-known results concerning the H-integral of single-valued functions (see \cite[Corollary 2F and Proposition 1C(a)] {fvol4}) and to Theorem \ref{5.6}, one also has easily
\begin{proposition}\label{additivita}
If  \ $F$ is {\rm H}-integrable, then for every $A, B \in \Sigma$ with $A \cap B = \emptyset$ it holds
\[ {\rm (H)}\int_{A \cup B} F \, d\mu  = {\rm (H)}\int_{A} F \, d\mu 
+ {\rm (H)}\int_B F \,  d\mu.\] 
\end{proposition}

The following Proposition compares the integral $J_E(F)$ (when it exists) with the integral $\Phi(F,E)$ of Definition 
\ref{normultintegral}.

 Later, in the Example \ref{nonH}, it will be shown that $\Phi(F,E)$ can be in $cwk(X)$ even when $F$ is not H-integrable.

\begin{proposition}\label{immersa con Phi}
Let  $F$ be {\rm H}-integrable.
Then, for every  $E \in \Sigma$ one has
$J_E(F)=\Phi(F,E),$ and therefore $\Phi(F,E)$ is in $cwk(X)$.
\end{proposition}
\begin{proof}
From the hypothesis one has that, for every integer $n$, a gauge $\gamma_n$ exists, such that
\begin{eqnarray}\label{deflim} 
\| i(\Sigma_{\Pi}F)-i(J_E(F)) \|_{\infty}\leq \frac{1}{n}
\end{eqnarray}
whenever $\Pi$ is a $\gamma_n$-fine partition of $E$, and so, thanks to Theorem \ref{5.6},
$J_E(F)\subset \Sigma_{\Pi}F+ n^{-1} B_X$:
indeed, from Theorem \ref{5.6} it follows that 
$$d_H( \Sigma_{\Pi}F,J_E(F))\leq n^{-1}.$$
This clearly implies that
$$J_E(F)\subset \bigcap_n\bigcup_{\gamma}\bigcap_{P_{\gamma,E}}
\left[\Sigma_{\Pi}F+n^{-1} B_X\right]=\Phi(F,E).$$
(This was also proved in (\ref{prima})).
On the other hand, from the formula (\ref{deflim}) it can be deduced that
$\Sigma_{\Pi}F\subset J_E(F)+n^{-1} B_X.$
Now, if $z\in \Phi(F,E)$, for the same integer $n$ a gauge
 $\gamma^{\prime}_n$ exists, such that $z\in \Sigma_{\Pi}F+ n^{-1} B_X$ for every $\gamma^{\prime}_n$-fine partition $\Pi$ of $E$. So, if one takes $\gamma^*_n:=\gamma_n\cap \gamma^{\prime}_n$, then, for every $\gamma^*_n$-fine partition $\Pi$ of $E$: 
$\Sigma_{\Pi}F\subset J_E(F)+n^{-1} B_X,\  z\in \Sigma_{\Pi}F +n^{-1} B_X,$
and therefore
$z\in J_E(F)+2n^{-1}B_X.$
Since this must be true for all $n$, and $J_E(F)$ is closed, it follows that $z\in J_E(F)$, and so the converse inclusion is proved. 
\end{proof}
Due to this fact one can consider $\Phi(F,E)$ as a multivalued set function, as soon as $F:T\to cwk(X)$ is 
H-integrable, and for  any Banach space $X$. In particular 
\begin{definition}\label{multisub} \rm \cite[Definition 3.1]{cascales2007a}
A multivalued set function 
 $M:\Sigma \rightarrow cwk(X)$ is a  {\em finitely additive} (respectively {\em countably additive})   multimeasure if
 $M(\emptyset) = {0}$ and 
 $M(A \cup B) = M(A)+ M(B)$ for every $A,B \in \Sigma$ , with $A \cap B = \emptyset$
(respectively if for every disjoint sequence $(E_n)$ in $\Sigma$ the series $\sum_n M(E_n)$ is unconditionally convergent and $M(\cup_n (E_n)) = \sum_n M(E_n)$).
\end{definition}

\begin{corollary}
Let $F: T \rightarrow cwk(X)$ be any 
{\rm H}-integrable multifunction.  Then, for every $E \in \Sigma$,  $M(E):=\Phi(F,E)$ is a countably  additive multimeasure.
Moreover, in the topology of $C(\Omega)$, $M$ is $\sigma$-additive and $\mu$-absolutely continuous.
\end{corollary}
\begin{proof}
The first part is an immediate consequence of Propositions \ref{immersa con Phi}, 
\ref{nuovoconfronto}  and   \cite[Theorem 4.1]{cascales2007a} since $F$ is Pettis integrable. 
As to the second part, thanks to  \cite[Proposition 2.3]{cs2015},
   \g\ is McShane integrable. So, by
\cite[Theorem 1Q]{f1995},   \g\ is Pettis integrable and therefore $M$ turns out to be $\sigma$-additive and $\mu$-continuous, thanks to well-known properties of the Pettis integral.
\end{proof}
Finally:
\begin{proposition}
If $F: T \to cwk(X)$ is {\rm H}-integrable then, for every $E \in \Sigma$ it is:
\[ {\rm(H)} \int_E F d\mu =  \overline{\left\{\int_E f d\mu, f \in S^1_{Pe}\right\}},
\]
(here, $S^1_{Pe}$ has a similar meaning as $S^1_{F,H}$, but with Pettis in the place of Henstock integrability)
while, if $X$ is reflexive
\[ {\rm(H)} \int_E F d\mu = 
%{\rm (AH)} 
\overline{\left\{\int_E f d\mu, f \in S^1_{F,H} \right\}}\]
\end{proposition}
\begin{proof}
 Observe that the Aumann integrals involved are non empty thanks to \cite[Theorem 2.5]{cascales2007a} and Theorem \ref{esistenza} respectively.
So the assertion is an obvious consequence of \cite[Theorem 4.3]{BS2011}
since $\mu$ is pointwise non atomic and reflexivity of the space avoids the hypothesis on free cardinals. 
\end{proof}
%======================================
\subsection{Order multivalued integrals}\label{o-multi}
Now, the multivalued integral in the {\em order} sense will be studied, with the following assumption:
\begin{description}
\item[$(H_0)$]
 {\bf $X$  is a weakly $\sigma$-distributive Banach lattice with an
order continuous norm, $\|\cdot \|.$}
\end{description}
Also in this subsection, a small survey is presented, comparing some different notions of integral in the {\em order} sense. Some new results will be obtained, similar to those already existing in the normed case: nevertheless, sometimes there are  differences, which
 may look at first sight not so deep as they really are.

\begin{definition}\label{starintegral}\rm (\cite[Definition 3.6, Proposition 3.7]{bms})
Let \mbox{$F:T \to 2^X$} be a multifunction with non-empty values, 
and $E \in \Sigma$. The \it $(o)$-integral \rm of $F$ on $E$ is the set
$$\Phi^o(F,E):= \bigcup_{(b_{n})_{n}, (\gamma_n)_n}  \bigcap_{n}  \bigcap_{\pgn}  {\cal U}\left ( \Sigma_{\Pi}(F) , b_n \right),$$
where $(b_n)_n$ denotes  any $(o)$-sequence, $(\gamma_n)_n$
 any sequence of {\em gauges}, $P_{\gamma_n}$ the family of all $\gamma_n$-fine
 partitions of $E$, and\,   ${\cal U}(C, b):= \{ z\in X\, \mbox{ such that}\,\, |z-y|\leq b\,\,  \mbox{for some}\,\,  y\in C\}$.
%==========================
\begin{comment}
\begin{eqnarray*}
\Phi^{o}(F,E) & ?=? & \left\{ \right. z \in X: \text{ there exist an $(o)$-sequence} \,\, (b_n)_n
\\ && 
\text{and a corresponding sequence} \,\, (\gamma_n)_n\ \text{of gauges}:\\ &&
\text{ for all } n \in \enne \, \, 
\text{and for every\,\,}
\gamma_n\text{-fine partition } \\ &&  P:=
\{ (E_i,t_i): i=1, \ldots,k \} 
 \, \text{of } E \text{ there exists }\\ &&
\, \displaystyle{ \, c \in \sum_{i \leq k} \, F(t_i) \, \mu(E_i)} \text{ with\,\,}
\displaystyle{|z-c| \leq b_n \, \left. \right\}}.
\end{eqnarray*}
\end{comment}
%========================
\end{definition}

\noindent Of course, when $F$ is single-valued, $F=\{f\}$,  the integral 
$\Phi^{o}(F,E)\equiv  {\rm (oH)} \int_E f d\mu$,
if non-empty, is a singleton, and in this case $f$ is  {\rm oH}-integrable. 
Moreover, for multivalued mappings $F:T\to cwk(X)$ and  for any measurable set $E$ the following inclusion holds, thanks to order-continuity of the norm:
$\Phi^{o}(F,E)\subset \Phi(F,E);$
in case $X$ is also an $M$-space, then also the reverse inclusion holds true as in \cite[Proposition 3.14]{bms}.

Similarly as for the norm integral, an alternative notion can be given, parallel to Definition \ref{normintegral}:

\begin{definition}\label{oHmultivoco}\rm
Let $F:T\to cwk(X)$ be any multifunction.  $F$ is {\em{\rm oH}-integrable} if,
for every measurable $E\subset T$ there exist an element $J_E\in cwk(X)$, an $(o)$-sequence $(b_n)_n$ in $X$ and a corresponding sequence $(\gamma_n)_n$ of gauges in $T$, such that, for every $n$ and every $\gamma_n$-fine partition $\Pi$ of $E$,\  \
$\Sigma_{\Pi}(F)\subset {\cal U}\left (J_E , b_n \right),\ \ {\rm and}\ \ J_E\subset {\cal U}\left (\Sigma_{\Pi}(F) , b_n \right).$
\end{definition}
\begin{proposition}\label{unicoint}
If $F:T\to cwk(X)$ is {\rm oH}-integrable, then the element
$J_E$ is unique.
\end{proposition}
\begin{proof}
 Without loss of generality,  assume that $E=T$, and suppose that $J:=J_T$ and $J':=J'_T$ are two elements of $cwk(X)$ satisfying the condition in Definition \ref{oHmultivoco}. Let $(b_n)_n$ and $(b'_n)_n$ be the $(o)$-sequences relative to $J$ and $J'$ respectively, and $(\gamma_n)_n$, $(\gamma'_n)_n$ be the corresponding gauges. Then, if we fix $n$, and take any $(\gamma_n\wedge \gamma'_n)$-fine partition $\Pi$, we get
$$J\subset  {\cal U}\left (\Sigma_{\Pi}(F) , b_n \right)\subset {\cal U}\left (J' , b_n+b_n' \right),$$
and also
$$J'\subset  {\cal U}\left (\Sigma_{\Pi}(F) , b'_n \right)\subset {\cal U}\left (J , b_n+b_n' \right).$$
So, for every element $a\in J$ and every integer $n$, there exists an element $a'_n\in J'$  such that $|a-a'_n|\leq b_n+b'_n$. This clearly means that the sequence $(a'_n)_n$ is $(o)$-convergent to $a$, and so also norm-convergent to $a$. Since $J'$ is closed, this implies $a\in J'$. So, $J\subset J'$. Reversing the role of $J$ and $J'$, also the converse inclusion is obtained, and then $J=J'$. 
\end{proof}
This integral is related to $\Phi^{o}$, in the following sense.

\begin{theorem}\label{confrontoPhistar}
Let $F$ be as above, and assume it is {\rm oH}-integrable. Then, for every measurable $E\subset T$,
$\Phi^{o}(F,E)=J_E,$
and so $\Phi^{o}(F,E)$ is in $cwk(X)$.
\end{theorem}
\begin{proof}
Again, the proof will be given only for the case $E=T$.
Let $(b_n)_n$ and $(\gamma_n)_n$ be the $(o)$-sequence and the corresponding sequence of gauges regulating oH-integrability of $F$. \\
Let $w$ be any arbitrary element of $J_T$, and fix $n$. Then, if $\Pi$ is any $\gamma_n$-fine partition, clearly 
$w\in {\cal U}\left (\Sigma_{\Pi}(F) , b_n \right).$
But this is precisely the condition that $w\in \Phi^{o}(F,T)$. By arbitrariness of $w$, one obtains the inclusion $J_T\subset  \Phi^{o}(F,T)$ (and also that  $\Phi^{o}(F,T)\neq \emptyset$).
\\
In order to prove the converse inclusion, take any element $z\in \Phi^{o}(F,T)$, and let $(b'_n)_n$ and $(\gamma'_n)_n$ be the $(o)$-sequence and the corresponding sequence of gauges related to the definition of $\Phi^{o}(F,T)$. Now, if $\Pi$ is any $(\gamma_n\wedge \gamma'_n)$-fine partition, 
$$z\in {\cal U}\left (\Sigma_{\Pi}(F) , b'_n \right)\subset {\cal U}\left (J_T , b'_n+b_n \right).$$
As above, this implies that $z$ is in the norm-closure of $J_T$, i.e. $z\in J_T$. By arbitrariness of $z$, this gives $\Phi^{o}(F,T)\subset J_T.$
This proves the reverse inclusion, and therefore  the announced equality. 
\end{proof}

Observe that, also for the order integral, the set $\Phi^o(F,T)$ can be in $cwk(X)$ even when $F$ is not oH-integrable, see Example \ref{nonH}.\\

\noindent {\bf From now on, order-boundedness of the sets $F(t)$ will be assumed, for any multivalued mapping $F$ and  every $t\in T$.}

\begin{remark} \label{semplici}\rm
Observe that, thanks to \cite[Theorem 3.12]{bms} if $F$ is a simple function with values in $cwk(X)$ (namely 
$F = \oplus_{i \leq n} C_i 1_{E_i}, C_i \in cwk(X), E_i \cap E_j = \emptyset, i \neq j,  i \leq n$) then
 $\Phi^{o}(F,E)$ is in $cwk(X)$ and 
$$
\Phi^o (F,E) = \sum_{i \leq n} C_i \mu (E_i) = \{{\rm (oH)}\int_E f d\mu: \
 f(t)\in F(t)\ \mu-a.e.\}
$$
In fact, in  the quoted result the equivalence is stated among the $\Phi^o$-integral and the {\em order}-closure of the Aumann-Henstock integral, when the $C_i$'s are in $cbf(X)$. But  $C_i \in cwk(X)$ so one has direct coincidence with the Aumann integral.
\end{remark}
A kind of selection theorem is stated now.
This result is parallel in some sense to the  Kuratowski-Ryll Nardzewski selection theorem, but related to the order structure of the space $X$ rather than to its topology. We shall see later also some consequences.
\begin{theorem}\label{quasiselezione}
Let $F:T\to cwk(X)$ be any {\rm oH}-integrable mapping, with integral $J$, and define
$g(t):=\sup {F(t)},\ \ S:=\sup{J}.$
Then, $g$ is {\rm oH}-integrable, and its integral is $S$.
\end{theorem}
\begin{proof}
 Let $(b_n)_n$ and $(\gamma_n)_n$ be
the sequences  
introduced in the Definition \ref{oHmultivoco}, and fix any $\gamma_n$-fine partition $\Pi$. Then 
\begin{eqnarray}\label{eraprima}
\sigma(g,\Pi)\in   {\cal V}\left (S , b_n \right)\end{eqnarray}
where  
$${\cal V}\left (A , b \right):=\{ z \in X :\, \exists\, a_0 \in A \,\, \mbox{with}\,\, z \leq a_0 + b\}, $$
   for every $
(A,b) \in  (cwk(X),X^{++}).$
Indeed, as $\Pi\equiv (t_i,I_i)$ is $\gamma_n$-fine, 
$\sum\alpha_i\mu(I_i)\leq S+b_n,$
for every choice of the points $\alpha_i\in F(t_i)$. Hence
$$\alpha_1\mu(I_1)\leq S+b_n-\sum_{i>1}\alpha_i\mu(I_i),$$
and, by varying just $\alpha_1$ one gets
$g(t_1)\mu(I_1)\leq  S+b_n-\sum_{i>1}\alpha_i\mu(I_i),$
from which
$$\sum_{i>1}\alpha_i\mu(I_i)\leq -g(t_1)\mu(I_1)+S+b_n.$$
Now, isolating $\alpha_2$ and proceeding in the same fashion, it follows
$$\sum_{i>2}\alpha_i\mu(I_i)\leq -g(t_1)\mu(I_1)-g(t_2)\mu(I_2)+S+b_n;$$
then it is clear that, continuing in this way,  (\ref{eraprima}) is proven.
\\
On the other hand, easily one has
$$J\subset {\cal V}\left (\Sigma_{\Pi}(F) , b_n \right)\subset {\cal V}\left(\sigma(g,\Pi) , b_n \right),$$
hence
$S\leq \sigma(g,\Pi)+b_n.$
So, there exist an $(o)$-sequence $(b_n)_n$ and a sequence $(\gamma_n)_n$ of gauges such that, for every $n$ and every $\gamma_n$-fine partition $\Pi$,
$\sigma(g,\Pi)\leq S+b_n,\  {\rm  and} \  S\leq \sigma(g,\Pi)+b_n$
i.e. $$|\sigma(g,\Pi)-S|\leq b_n,$$
from which the assertion follows. 
\end{proof}
In some cases, one can obtain a decomposition  similar to Theorem \ref{nuovoconfronto}.4), for {\rm oH}-integrable multifunctions (see also Theorem \ref{deco})
\begin{theorem}\label{deco2}
Let $F:T\to cwk(X)$ be any {\rm oH}-integrable function, such that
\begin{description}
\item[(\ref{deco2}.1)]
 $\sup{F(t)}\in F(t)$ for each $t\in T$. 
\end{description}
Then $F$ is the sum of an {\rm oH}-integrable single-valued mapping $g:T\to X$ and an {\rm oH}-integrable mapping $G:T\to cwk(X)$ such that $s(x^*,G(t))\geq  0$ for all elements $x^*\in X^*$ and $s(x^*,G(t))=0$ for all positive elements $x^*\in X^*$.
\end{theorem}
\begin{proof}
 Let $g(t)=\sup{F(t)}$ for all $t$, and define $G(t)=F(t)-g(t)$ by translation. Then clearly $\sup{G(t)}=0$. Moreover, from (\ref{deco2}.1) it follows $0\in G(t)$, by a translation argument. Now, for every fixed $t$ and any element $x^*\in X^*$, clearly $s(x^*,G(t))\geq 0$.
\\
In case $x^*$ is positive, we also have $0=x^*(0)\geq x^*(u)$ for all $u\in G(t)$, so $0\geq s(x^*,G(t))$. Combining this result with the previous one, we get $s(x^*,G(t))=0$ for all positive $x^*$.
\\
It only remains to show that $G$ is {\rm oH}-integrable. Indeed, it will be proved that its integral is $J-\sup\{J\}$, where $J={\rm(oH)}\int_T F d\mu$.
By integrability of $F$ and $g$, there exist an $(o)$-sequence $(b_n)_n$ in $X$ and a corresponding sequence $(\gamma_n)_n$ of gauges, such that, as soon as $\Pi$ is any $\gamma_n$-fine partition, one has
$$\Sigma_{\Pi}(F)\subset {\cal U}(J,b_n),\ \ J\subset {\cal U}(\Sigma_{\Pi}(F),b_n), \ \  |\sigma(g,\Pi)-\sup{J}|\leq b_n.$$
From this, it is easy to see that
$$\Sigma_{\Pi}(G)\subset {\cal U}(J-\sup{J},2b_n),\ \ J-\sup{J}\subset {\cal U}(\Sigma_{\Pi}(G),2b_n).$$
This suffices to prove integrability of $G$. 
\end{proof}

Observe here that condition (\ref{deco2}.1) is fulfilled, for example, if $F(t)$ is upwards directed for every $t$: in this case $\sup{F(t)}\in F(t)$ thanks to \cite[Proposition 354 E]{fvol3}.\\

The previous results can be used to introduce an example, in which the set $\Phi^o(F,T)$ is in $cwk(X)$, but the mapping $F$ is not (oH)-integrable.
\begin{example}\label{nonH}\rm
Let $T=[0,1]$ and $A\subset [0,1]$ be any non-measurable set, so that the mapping $1_A$ is not Lebesgue integrable. Clearly, this also means that $1_A$ is not H-integrable (in any sense, both in order and topology, since here the space $X$ is $\erre$). Next, define the following multivalued map $F:[0,1]\to ck(\erre)$ as follows:
$$F(t)=[0,1_A(t)],$$
for every $t\in [0,1]$ (here, of course, $[0,0]$ means $\{0\}$).
\\
Clearly, $F$ is an order-bounded map, and the element $0$ belongs to all the sums $\Sigma_{\Pi}(F)$, for whatsoever partition $P$ of $[0,1]$. Hence $\Phi^o(F,T)=\Phi(F,T)$ is non-empty and compact, thanks also to \cite[Proposition 3.8]{bms}.
\\
But $F$ is not (oH)-integrable, otherwise the mapping $t\mapsto \max F(t)$ would be (oH)-integrable by Theorem \ref{quasiselezione}, which is impossible.
\\
This example shows that the quantity $\Phi^o(F,T)$ might
replace the integral in some cases of non-integrability, when the previous Theorem \ref{confrontoPhistar} is not valid.
\end{example}

For single-valued mappings,  \cite[Theorem 15]{cs2014} shows that,
at least in $L$-spaces, {\rm oH}-integrability implies Bochner integrability (with the same integral). 
Observe also that this result is somewhat similar to  \cite[Theorem 5.12]{gould}. 
A consequence of  \cite[Theorem 15]{cs2014}
is the following result, which can be viewed as a particular version of Theorem \ref{deco2}.
\begin{corollary}\label{colsup}
Let  $L$ be any $L$-space, and $F:T\to cwk(L)$ be an {\rm oH}-
integrable mapping. Assume also that $\sup{F(t)}\in F(t)$ for all $t$. Then $F$ is the sum of a Bochner integrable single-valued mapping $f:T\to L$ and an {\rm oH}-integrable mapping $G:T\to cwk(L)$ such that $s(x^*,G(t))\geq 0$ for all elements $x^*\in L^*$ and $s(x^*,G(t))=0$ for all positive elements $x^*$.
\end{corollary}
\begin{proof}
 It is enough to combine Theorem \ref{deco2} with \cite[Theorem 15]{cs2014}. 
\end{proof}

Some conditions will be given now, ensuring oH-integrability of a multivalued function. The first is a Lemma of the Cauchy-type.
\begin{lemma}\label{ocauchy}
Let $F:T\to cwk(X)$ be any set-valued mapping. 
Then $F$ is {\rm oH}-integrable if and only if 
there exists an $(o)$-sequence $(b_n)_n$ in $X$ and a corresponding sequence of gauges $(\gamma_n)_n$  such that for every n, and every pair $\Pi$, $\Pi'$ of $\gamma_n$-fine partitions of $T$ one has
$\Sigma_{\Pi}(F)\subset \mathcal{U}(\Sigma_{\Pi'}(F), b_n).$
\end{lemma}
\begin{proof}
The necessary condition is obvious. For the converse implication observe that 
$\mathcal{U}(C, b)=C+[-b,b]$
for all sets $C\subset X$ and all $b\in X^{++}$. 
Since $ [-b,b]\subset \|b\|B_X$ for all $b\in X^{++}$, and $(b_n)_n$ is an $(o)$-sequence, the condition above implies that 
$d_H(\Sigma_{\Pi}(F),\Sigma_{\Pi'}(F))\leq \|b_n\|$
(and $\lim_n\|b_n\|=0$). This is precisely the Cauchy condition for the (H)-integral of $F$, and this implies that $J:=(H)\int_T F d\mu$ exists, thanks to Theorem \ref{nuovoconfronto}. Now, in order to prove that $J$ is also the (oH)-integral of $F$, 
fix $n$ and let $\Pi$ be any $\gamma_n$-fine partition. Moreover, for each $\varepsilon>0$ there exists a gauge $\gamma'$ such that
$d_H(\Sigma_{\Pi'}(F),J)\leq \varepsilon$
for any $\gamma'$-fine partition $\Pi'$. In particular, if  $\Pi'$ is $\gamma'\wedge \gamma_n$-fine, one has
$$\Sigma_{\Pi}(F)\subset \Sigma_{\Pi'}(F)+[-b_n,b_n]\subset J+[-b_n,b_n]+\varepsilon B_X.$$
Now, since $[-b_n,b_n]$ is closed (see \cite[Lemma 354B(c)]{fvol3}) and $J$ is weakly compact, $J+[-b_n,b_n]$ is closed too. So, by letting $\varepsilon$ tend to $0$, it follows easily
$$\Sigma_{\Pi}(F)\subset J+[-b_n,b_n]$$
for any $\gamma_n$-fine partition $\Pi$.
A perfectly symmetric reasoning proves also the reverse inclusion:
$$J\subset \Sigma_{\Pi}(F)+[-b_n,b_n]\,\,\,
\mbox{i.e.} \,\,\, J={\rm (oH)} \int F d\mu.$$
\end{proof}

The next result is inspired at \cite[Lemma 5.35]{KS}, and will be applied later.

\begin{proposition}\label{maggioriminori}    
Let $F:T\to cwk(X)$ be any set-valued mapping. Assume that there exists an $(o)$-sequence $(b_n)_n$ in $X$ such that, for every $n$ a couple of {\rm oH}-integrable mappings $G_1,G_2$ exist, from $T$ to $cwk(X)$, such that  $G_1(t)\subset F(t)\subset G_2(t)$ for every $t\in T$, and {\rm (oH)}-$\int_T G_2 d\mu \subset \mathcal{U}(\int_T G_1  d\mu,b_n)$.
Then $F$ is {\rm oH}-integrable.
\end{proposition}
\begin{proof}
Let $(b_n)_n$ be an $(o)$-sequence as in the hypothesis. Let also $(\beta_n)_n$ and $(\gamma_n)_n$ be an $(o)$-sequence and its corresponding sequence of gauges regulating oH-integrability of $G_1$ and $G_2$ (without loss of generality they can be taken the same for both multifunctions). Now, fix $n\in \enne$ and take two $\gamma_n$-fine partitions $\Pi$ and $\Pi'$ of $T$. Then, denoting by $J_1$ and $J_2$ the integrals of $G_1$ and $G_2$ respectively, one has
$$\Sigma_{\Pi}(F)\subset \Sigma_{\Pi}(G_2)\subset J_2+[-\beta_n,\beta_n]\subset J_1+[-b_n-\beta_n,b_n+\beta_n]$$
and
\begin{eqnarray*}
J_1+[-b_n-\beta_n,b_n+\beta_n] &\subset& \Sigma_{\Pi'}(G_1)+[-b_n-2\beta_n,b_n+2\beta_n]\subset \\
&\subset& \Sigma_{\Pi'}(F)+[-b_n-2\beta_n,b_n+2\beta_n].
\end{eqnarray*}
So, comparing the last two formulas, and taking $c_n:=b_n+2\beta_n$, one gets 
$$\Sigma_{\Pi}(F)\subset \mathcal{U}(\Sigma_{\Pi'}(F),c_n)$$
for all $\gamma_n$-fine partitions $\Pi, \Pi'$. Integrability now follows from Lemma \ref{ocauchy}. 
\end{proof}

%====================================================
%%====================================================

%===============================================================
%===============================================================
\section{The $[0,1]$ interval case}\label{01multi}
In this section, one considers functions defined on the unit interval $[0,1]$, endowed with the Lebesgue measure $\lambda$, and taking values in an arbitrary Banach lattice with order-continuous norm.
In order to define Henstock integrability (and integral), the only partitions allowed consist of (pairwise non-overlapping) subintervals of $[0,1]$. 
This produces the well-known distinction between Henstock and McShane integrability: indeed, in the latter type, the partitions allowed still consist of subintervals, but the tags need not belong to the corresponding subintervals. However, if $f$ is McShane-integrable in this sense, then the integral can be equivalently defined by allowing also partitions consisting of general measurable subsets: see \cite[Proposition  1F]{f1995}.
For this reason, McShane integrability in this sense will be still mentioned as oH-integrability.\\

 A first result is that, in this case, monotonicity implies oH-integrability:
 the following result is inspired at \cite[Example 5.36]{KS}.

\begin{theorem}\label{monotone} 
Assume that $F:[0,1]\to cwk(X)$ is an increasing mapping, with respect to the inclusion. Then $F$ is {\rm oH}-integrable.
\end{theorem}
\begin{proof}
 Since $F(t)$ is order-bounded for all $t$,  one can set
$K:=\sup\{|x|: x\in F(1)\}$ and obtain that $F(t)\subset [-K,K]$ for all $t$. 
Now, for each integer $n$ let $t_i:=i n^{-1}$, $i=0,...,n$, and define two multivalued mappings, $G_1$ and $G_2$, in the following way:
\begin{eqnarray*}
G_1(t) &=& \left\{\begin{array}{ll}
                        F(t_i),& t\in [t_i,t_{i+1}[, i=0,1,...,n-1\\
                        F(t_{n-1}),&  if\ \ t=1;
\end{array}
\right.
\\
G_2(t) &=& \left\{ \begin{array}{ll}
                        F(t_{i+1}), \ \ t\in ]t_i,t_{i+1}], i=0,1,...,n-1\\
                        F(0), \ if \ \ t=0.
\end{array}\right.
\end{eqnarray*}
Clearly $G_1$ and $G_2$ are oH-integrable since they are simple
(see Remark \ref{semplici}), and it is obvious that $G_1(t)\subset F(t)\subset G_2(t)$ for all $t$.
Now, it will be  proven that 
$\int G_2 d\lambda \subset \int G_1 d\lambda+[-2K n^{-1}, 2K n^{-1}]:$
thanks to Proposition \ref{maggioriminori} this will yield integrability of $F$.
Of course, 
$$\int G_2 d\lambda=n^{-1} \sum_{i=1}^n F(t_i), \ \ \ 
\int G_1 d\lambda = n^{-1} \sum_{i=1}^n F(t_{i-1}).$$
Now, take any element $z\in \int G_2 d\lambda$: then 
$z=n^{-1}(x_1+x_2+...+x_n),$
for suitable elements $x_i\in F(t_i), i=1,...,n$. Let us choose arbitrarily $x_0\in F(0)$ and define
$w:=n^{-1}(x_0+x_1+x_2+...+x_{n-1}).$
Of course, $w\in \int G_1 d\lambda$ and
$|z-w|=n^{-1}(|x_n-x_0|)\leq 2K n^{-1}.$
In conclusion, for every element $z\in \int G_2 d\lambda$ there exists an element $w\in \int G_1 d\lambda$ such that $|z-w|\leq 2K n^{-1}$,
i.e.
$$\int G_2 d\lambda \subset \int G_1 d\lambda +[-2Kn^{-1}, 2K n^{-1}]$$
as announced. The proof is now complete. 
\end{proof}

Observe here that the last Theorem is not a consequence of \cite[Theorem 19]{cs2014} and Theorem \ref{nuovoconfronto}: 
in fact, combining these two results one only obtains norm-integrability of $F$.
\section*{Conclusions}

In this paper the notions of Henstock (Mc Shane) integrability for
 functions and multifunctions defined in a metric compact regular space
 and taking values in a Banach lattice with an order-continuous norm are
 investigated. Both the norm-type and the order-type integrals have been
 examined. 
 Thanks to the structure of near vector space of $cwk(X)$, it is proved that the
 Henstock norm integrability of a multifunction $F$ is equivalent to the
 Henstock integrability (both in the norm and in the order sense) of its
 embedded function $i(F)$.
 Selections and decomposition results are also obtained.\\
Also the case of $ck(X)$-valued mapping can be discussed by using an analogous construction, with similar conclusions.

\section*{Acknowledgements}
The authors have been supported by University of Perugia -- Department of Mathematics and Computer Sciences -- Grant Nr 2010.011.0403, 
Prin: Metodi logici per il trattamento dell'informazione,   Prin: Descartes and by the Grant prot. U2014/000237 of GNAMPA - INDAM (Italy).
%=======================================================================================
%=======================================================================================

\end{document}